\documentclass[12pt]{article}
\usepackage[utf8]{inputenc}
\usepackage{amsmath}
\usepackage{amsthm}
\usepackage{graphicx}
\usepackage{amssymb}
\usepackage{amsfonts}
\usepackage{romannum}
\usepackage[final]{changes}
\newtheorem{thm}{Theorem}
\newtheorem{lem}{Lemma}

\newcommand{\blank}[1]{\hspace*{#1}}

\begin{document}

\setcounter{section}{0} 
\pagenumbering{arabic}
\setcounter{page}{1}
\newtheorem{lemma}{Lemma}
\newtheorem{theorem}{Theorem}
\newtheorem{remark}{Remark}[section]
\newtheorem{corollary}{Corollary}[section]
\newtheorem{proposition}{Proposition}
\newcommand{\vect}[1]{\overline{#1}}
\def\e{\eta}
\def\m{\mu}
\def\a{\alpha}
\def\r{\rho} \def\s{\sigma}
\def\l{\lambda}
\def\P{\mbox{P}}
\def\E{\mbox{\rm E}}
\def\Cov{\mbox{\rm Cov}}
\def\o{\omega} \date{}
\title{When to arrive at a queue with earliness, tardiness and waiting costs}
\author{Eliran Sherzer and Yoav Kerner}

\newcommand{\overbar}[1]{\mkern 1.5mu\overline{\mkern-1.5mu#1\mkern-1.5mu}\mkern 1.5mu}

\maketitle

\begin{abstract}
\noindent We consider a queueing facility where customers decide when to arrive. All customers have the same desired arrival time (w.l.o.g.\ time zero). There is one server, and the service times are independent and exponentially distributed. The total number of customers that demand service is random, and follows the Poisson distribution. Each customer wishes to minimize the sum of three costs: earliness, tardiness and waiting. We assume that all three costs are linear with time and are defined as follows. Earliness is the time between arrival and time zero, if there is any. Tardiness is simply the time of entering service, if it is after time zero. Waiting time is the time from arrival until entering service. We focus on customers' rational behaviour, assuming that each customer wants to minimize his total cost, and in particular, we seek a symmetric Nash equilibrium strategy. We show that such a strategy is mixed, unless trivialities occur. We construct a set of equations that its solution provides the symmetric Nash equilibrium. The solution is a continuous distribution on the real line.  We also compare the socially optimal solution (that is, the one that minimizes total cost across all customers) to the overall cost resulting from the Nash equilibrium.
\end{abstract}

\section{Introduction}

In many real life situations, customers choose strategically when to
arrive to a queue. For example, when commuting to work in the morning, or going to
 the cafeteria at lunch time, etc. In those kind of
situations, customers usually have a preferred arrival time,
such as, commuting to work right after dropping the kids at school. For the sake of simplicity, we
refer to the preferred arrival time as time 0 (w.l.o.g).
In many cases, customers share the same preferred arrival time.
This is very likely to be the case in the morning commute, or when going out for lunch.
If all customers wish to arrive at that time, it causes congestions.
As a consequence, one would consider alter his arrival time in order
to avoid waiting for a long time. In this situation, from an individual
point of view, a trade-off between the arrival time preference and
the queue length is formed. This trade-off was studied widely in the
literature
 (see~\cite{Breinbjerg},~\cite{Haviv13},~\cite{junejashimkin11},~\cite{junejajain09}, and~\cite{Ravner14}). Whereas, in all of them the trade-off was between waiting and tardiness
 costs. However, we suggest an additional cost, that is also reflected in the dilemma of when to arrive, and it is referred as \emph{earliness cost}.
This cost is induced by making the effort of arriving earlier than
desired. A classic example would be
 commuting to work at 4 AM just to avoid traffic jams. In this example, clearly, arriving this early minimizes the waiting time in one hand,
  but on the other, commuting so early is inconvenient. We further discuss
 the motivation of adding earliness cost in the next section.  \\
\indent In this study, we propose a model where customers choose
when to arrive. Each individual possess three types of costs:
Earliness, tardiness and waiting. Whereas, waiting refers to the amount of
time one waited in the queue. Earliness is related to the amount of time one
arrived before time 0, and if he arrived after time 0, then there is
no earliness cost. Tardiness is related to the amount of time entering service after time 0. However, if an individual service time began
prior to time 0, then there is no tardiness cost. We focus on
homogenous customers with respect to their costs and desired service
time. Under this assumption, customers make their arrival decision
while knowing that others have the same preferences. In such
systems, all customers' decisions interact. That is, an individual's
cost is determined not only by his own decision, but also by
others', and thus the natural model of customers' arrivals is a
non-cooperative game. Motivated by this point of view, we look for a
symmetric Nash equilibrium. That is, if a (possibly) mixed arrival
strategy is used by all, an individual has no incentive to
unilaterally deviate from it. This solution, which determines the
arrival process for the system, implies the opening and closing
times of the system. We also discuss a system with exogenous opening
and closing times, and study the impact of these constraints on
customers' behavior.\\
\indent We found that in the case where there were no opening and
closing time constraints, the Nash equilibrium is mixed and unique,
and has a positive density $f(t)$ along the interval
$[-T_{e_1},T_{e_2}]$, with some $T_{e_1}, T_{e_2}>0$. $f(t)$ is
complete (that is, the CDF has no jumps) and continuous almost
everywhere, with a connected support. In the case where the opening
and closing times are constrained, the equilibrium may be either
pure or mixed, and if the latter is the case, then it can also
contain an atom.\\
\indent In this paper we show how to compute the Nash equilibrium
strategy for all of the above-mentioned cases. The equilibrium
strategy, in general, is not \added{socially} optimal, in the sense
that it doesn't minimize the expected overall costs of all the
customers. This is because the decision of when to arrive, taken by
an individual, imposes additional losses on the others. The latter
are often referred to as externalities. Therefore, finding the
socially optimal solution is also of our interest. This allows us
to compare the strategy that holds for the Nash equilibrium and the
one that minimizes the overall cost. In order to compare
\deleted{between}the two, we derive their ratio, which is called the
\emph{price of anarchy} (PoA). The value of the PoA can help us
understand how much customers can \replaced{reduce}{minimize} their
costs by cooperating with each other.\\
\indent The analysis in the paper is performed for two types of
environment\added{s}. The first is the familiar stochastic
environment.
 In this case we assume that the number of users is random and follows
the Poisson distribution\footnote{One can think of a huge amount of potential users, each one participate with a significantly small probability.
 The Poisson approximation of the binomial  distribution leads to our assumption\ (see more in~\cite{Haviv13})}
 and the service time\added{s} are  \replaced{are}{follows the} exponential\added{ly} distribut\replaced{ed}{ion}, independent and assumed to be work conserving.
  The second environment is fluid-based, and each user is associated with a drop of infinitesimal size. The stochastic model is more accurate
   but can provide only numerical results (and not analytical).\ The fluid model on the other hand can provide analytic results. Moreover, $PoA$ can be obtained as well.
   We make some observations regarding the relations between the two models later on.\\
\indent The paper is organized as follows \replaced{:}{.} In
Section~\ref{sec:motivationliterture} we describe the model
motivation and provide a literature review. In
Section~\ref{sec:model} we present the model under study. In
Section~\ref{sec:poisonnocons}, we obtain the arrival equilibrium
strategy for both the stochastic model and the fluid model. In
addition, we compare via the fluid model the social cost of the
equilibrium arrival strategy and the socially optimal cost. In
Section~\ref{sec:poisoncons}, we obtain the above solution concepts
for the constrained model. Finally, in Section~\ref{con:conclusions}
we summarize our main results.

\section{Motivation and literature review}\label{sec:motivationliterture}

We next provide a motivation for the model under study. More specifically, we
explain why adding earliness cost to our model is important. This is
followed by a literature review.

\subsection{Motivation}\label{subsec:motivaition}

In the previous section, we introduced the concept of earliness
cost. Whereas, earliness cost induced by making the effort of
arriving earlier than desired. We next specify why we think it is
vital to include it in the model. For this purpose, we consider the
morning commute example. When commuting, it is very natural having a
preferred arrival time. For the sake of demonstration, we consider
an employee, that doesn't want to wake up too early for work in one hand, but on the other hand, doesn't want to
arrive too late, when there are important meetings to attend. For this
particular example, we set the desired arrival time to be 08:00 AM.
The problem is, arriving at 08:00 comes with heavy traffic. As a
consequence, each arrival point, is costly, due to the different cost types. In fact,
each point comes with an aggregation of the three cost types. For illustration, we consider a few arrival epochs.
Arriving between 05:30 to 06:30 AM induces
an expected high earliness cost, low waiting cost and no tardiness cost at all, since heavy traffic is not expected and there is no reason to be late in this case.
 Arriving between 07:30 to 08:30 AM induces a high waiting cost due to heavy traffic, and relatively low earliness and tardiness costs.
  Whereby, arriving between 10:30 to 11:30 AM, induces an expected high tardiness cost, no earliness cost since no effort to arrive earlier was made, and probably low expected waiting cost as the traffic is probably flowing normally at that point. \\
\indent The main idea is that when one needs to decide when to
arrive, there are three type distinctive types of costs to consider.
Moreover, they all should be under consideration simultaneously, as
they are all time-dependant. To the best of our knowledge, we are the first to
identify this additional cost (i.e. earliness) under a game theoretic environment\footnote{Vickrey, already in 1969, considered the earliness
concept in~\cite{Vickery69}, but not as a game.}. We further
suggest, that even in the proposed cafeteria problem
in~\cite{junejashimkin11},~\cite{junejajainshimkin11}
and~\cite{junejajain09}, one can also consider earliness cost. It is
assumed in those studies that customers want to arrive to eat lunch
when the cafeteria opens. But, it's very likely that customers don't
arrive early to the cafeteria since it is not convenient for them (e.g. not hungry yet)
and not because it is not open, as in many cases, they are open all day long. This inconvenience induces earliness cost
 and it should be distinctive to other two cost types, waiting or tardiness. Further examples which describe the proposed arrival dilemma are,
going to the supermarket, the gym and to the playground with the kids. Those places are open most of the time, but less crowded in the beginning or the end of the day, when it is less convenient to arrive.  \\
\indent The earliness cost can be interpreted in two ways. The first
one is simply arriving too early, for example, by waking up at the
break of dawn just to avoid traffic jams. The second is entering or
completing service earlier than desired, for example, arriving
really early to a concert in order to get good seats. However, in
this study, we focus on the first option, where it does seem more
natural in many scenarios of strategic arrivals. We take for example
the morning commute dilemma. If one made the effort and arrived
earlier than he originally planned, he really has no incentive to
stay longer on the roads just to arrive to work at his desired time.
That is, he is clearly better off arriving earlier to work rather
than being stuck in traffic. However, the second option suggest that
he might prefer waiting over arriving earlier to work, where
obviously it is not the case\footnote{We refer the reader
to~\cite{ravnerhassin17} for a model with the second option.}.
\subsection{Literature review}\label{subsec:literturereview} Similar
queueing problems, in which customers need to decide when to arrive,
have been introduced in recent studies. The first such study was by
Glazer and Hassin \cite{Glazer&hassin83}, and they considered a
model where service has an opening time point, which is set
\replaced{to be}{at} time 0 (w.l.o.g). However, customers are
allowed to arrive before time 0, thereby being referred to as early
birds. In order to be served, customers must arrive prior to a
predetermined closing time, which is denoted by $T$ and assumed to
be positive. Similar to our model the total number of arrivals comes
from a Poisson distribution.
 All customers wish to minimize their waiting time in
the queue. In that study it was shown that there is a unique, mixed
Nash equilibrium in the interval $(-w,T)$ for some $w>0$. It is also
shown that the arrival distribution in the interval $[-w,0]$ is
uniform, and with a positive density distribution $f(t)$ along
$(0,T]$. The authors showed that the expected waiting time for
arriving any time within the arrival support is $w$. This is because
the expected waiting time is constant along the support and the
queueing time of a customer who arrives at time $-w$ is exactly $w$,
since he is guaranteed to encounter an empty queue. A subsequent
study by Hassin and Kleiner \cite{Kliner&hassin11} presented the same
model, only with an additional restriction: arriving at $t<0$ is
forbidden. They showed that unless the system is heavily loaded, the
restriction does not change the average waiting time in a
significant way. They also compared the equilibrium solution to the
socially optimal solution using fluid models. They found that the
ratio between them increases when the system becomes more heavily
loaded. Another important aspect of this study was gaining a lower
equilibrium social cost, while restricting the arrival times to two
or three discrete points.\\
\indent Until this point, only the waiting cost
was considered. Juneja and Jain \cite{junejajain09} introduced the
concept of tardiness cost, which is added to the waiting cost in
their model. Their model was inspired by a situation where customers
decide when to arrive to a concert or a cafeteria. Each individual's goal is to arrive at the time $t$
that minimizes the expression $\mathbb{E}[\alpha
W(t)+\beta(t+W(t))]$, where, $\alpha$ and
$\beta$ were set to be the linear waiting and tardiness costs per
time unit respectively, $t$ is the arrival time and $W$ is the
waiting time.  Using fluid models, the resulting equilibrium
was shown to be a uniform arrival distribution in the interval
$[\tau_0,\tau_1]$ for some $\tau_0<0$ and $\tau_1>0$. They also
showed that the PoA is 2. This model differs from ours in two things. The first is that it doesn't consider tardiness cost and the second that they consider a constant number of arrival, as opposed to our model where it comes from a Poisson distribution. A subsequent and more general study was
published by Haviv \cite{Haviv13}. As in~\cite{junejajain09}, this
study recognizes tardiness cost as well as waiting cost. However,
the generalization of Haviv's study is due to including both the
case where early birds are allowed and the case where they are not.
In order to get service, customers are obliged to arrive prior to a
predetermined time $T>0$ ($T=\infty$ is possible). If early birds
are allowed, the obtained mixed equilibrium is a positive almost
continuous arrival probability density function along the interval
$(-w,T_e)$ for some $w>0$ and $T_e>0$, if $T_e<T$. The arrival
distribution is uniform for the interval $(-w,0)$. If early birds
are not allowed, the result is an atom of a size denoted by $p$ at
zero, an interval $(0,t')$ with a zero density, and some positive
density $f(t)$ along $[t',T_e]$ for some $T_e<T$, or a pure
equilibrium at time 0. However, when $T_e>T$ this means that the
pre-imposed closing time will affect the system. Of course this will
affect the solution, and now the upper limit of the arrival
distribution support $T_e$ is replaced by $T$. This model is almost identical to ours except we also consider earliness cost.  \\
\indent As mentioned, Juneja
and Jain based their study on the concert/cafeteria problem. Other
studies concerning this problem include \cite{junejashimkin11} and
\cite{junejajainshimkin11}, where the number of arrivals is finite
and non-random, and one by Honnappa and Jain~\cite{jainhonnappa12}
which considered a multiple-server model in which customers'
decision is to choose a server. Another similar model was used
in~\cite{Ravner14}, which considered a model including waiting time
along with penalties by index of arrival, and characterized the
arrival process, which constitutes a symmetric Nash equilibrium. \\
\indent Another study which include both waiting and tardiness cost was done
by Breinbjerg~\cite{Breinbjerg}. In this study as opposed to the
previous ones, nonlinear cost functions were considered and
the service distribution is general. By using a novel approach he
showed that there exists a unique symmetric Nash equilibrium along
numerical examples. Beside the theoretical studies, there are
several papers which includes empirical research. This allows to
examine the accuracy of analytical equilibrium solution. Strategic
queuing models with discrete arrival instances and deterministic
service durations, the theoretical predictions were compared to
empirical finding in controlled laboratory experiments, which
provide support for the symmetric mixed-strategy equilibrium
solution on the aggregate level (although not on the individual
level), see (\cite{Breinbjergsebald},~\cite{Seale05},~\cite{Stein07}
and~\cite{Rapoport04}).\\
\indent More related work was done by McAfee and
McMillan \cite{mcafeemcmillan87}, and Haviv and Milchtaich
\cite{Havivmilchtaich2012}. Their contribution is in games,
particularly in auctions with a random number of players. Haviv,
Kerner and Kella~\cite{Kerner09} described a related study as well,
in which there is a Poisson stream of customers to an $M/M/N/N$ loss
system. Platz and Østerdal~\cite{Platz2012} studied a class of
queueing games with a continuum of players and no possibility of
queueing before opening time, and show that the first-in-first-out
discipline induces the worst Nash equilibrium in terms of
equilibrium welfare, while the last-in-first-out queue discipline
induces the best. Another related set of papers concerns
transportation. Vickrey \cite{Vickery69} presented a deterministic
model where commuters have full information about other commuters,
and are able to change their arrival times in order to avoid heavy
traffic. The model assumes a single bottleneck and a constant number
of participants. Delay occurs only when the traffic flow exceeds the
capacity of the bottleneck. Each commuter has a preferred time to
pass the bottleneck. Just as in our model, there is a cost for
arriving earlier than desired, and a tardiness cost. Later studies
were published by Lago and Daganzo \cite{LagoDaganzo07}, Arnott
\cite{Arnottpalma99}, Ostubo and Rapoport \cite{OstuboRapoport08},
and their bibliographies.

\section{The models}\label{sec:model}
In this section we present the details, including cost functions, of both the stochastic and the fluid models.
\subsection{The stochastic model}
We consider a single-server queue in which customers decide when to
arrive. The service times, $X_1,X_2,...$, are independent and
exponentially distributed with parameter $\mu$ and assumed to be work conserving. The server is active
between predetermined opening and closing times, denoted by $-T_1$
and $T_2$ respectively, where  $T_1,T_2\in (0,\infty]$. The total
number of customers, follows a Poisson distribution with mean $\l$. Each customer wishes to minimize three types of
costs: earliness, tardiness and waiting. Earliness is the amount of
time from arrival till \replaced{0}{entering service}, and is
applicable only if the arrival time is before time 0. Tardiness is
the amount of time after time zero that a customer enters service.
Waiting is the time that a customer spends in the queue. All costs
are linear in time, with slopes $\beta_1$, $\beta_2$ and $\alpha$,
for earliness, tardiness and waiting, respectively.  To summarize,
the total cost for a customer who arrives at time $t_1$ and enters
service at time $t_2$ is \deleted{:}  \begin{equation*}
Cost\left(t_1,t_2\right)=
\left\{\begin{array}{cc}-\beta_1t_1+\alpha(-t_1+t_2),&t_1<t_2\leq0,
\\-\beta_1t_1+\alpha(-t_1+t_2) +\beta_2t_2,&t_1\leq
0<t_2,\\\alpha(-t_1+t_2)+\beta_2t_2,&0<t_1<t_2.\end{array}\right.
\end{equation*} Of course, the best thing for a customer would
be to arrive at time 0 and encounter an empty system. That way,
there will be no cost at all. We assume that the decision is taken
{\it prior} to observing the state of the queue. Therefore, while
making a decision, the only known variable\added{s}
\replaced{are}{is} the arrival time, and the marginal expected
waiting and tardiness times. Let $w(t)$ and
$C(t)$ be the expected waiting time and the expected total cost,
respectively, of a customer who arrives at time $t$. We have
\begin{align}
   C(t)=  \begin{cases}
              -t \beta_1 +\alpha w(t) + \beta_2 \mathbb{E}\left(\left[\sum\limits_{i=1}^{N(t)}X_i+t\right]^+\right),      & t<0,   \\
              (\alpha +\beta_2)w(t) + \beta_2 t,       &t\geq 0. \label{eq:cost}
            \end{cases}
\end{align}
 where  $N(t)$ is the number of customers in the system at time $t$. In Lemma~\ref{lem:Ex+t} we present $\mathbb{E}\left(\left[\sum\limits_{i=1}^{N(t)}X_i+t\right]^+\right) $ in terms of the model parameters. \begin{lem}\label{lem:Ex+t}
\begin{align}
 &\mathbb{E}\left(\left[\sum\limits_{i=1}^{N(t)}X_i+t\right]^+\right)  =  \sum_{j=0}^{\infty}p_j(t)\int_{s=0}^{\infty}\sum_{k=0}^{j-1}\frac{1}{k!}e^{-\mu (s-t)}(\mu(s-t))^{k}ds, \label{eq:E(w(t)+t))}
\end{align}
where $p_j(t)=P(N(t)=j)$ is the probability of having $j$ customers
at time $t$.
\end{lem}
\begin{proof}
\begin{align*}
&\mathbb{E}\left(\left[\sum\limits_{i=1}^{N(t)}X_i+t\right]^+\right)  =  \sum_{j=1}^{\infty}p_j(t)\mathbb{E}\left(\left[\sum\limits_{i=1}^{j}X_i+t\right]^+\right)\\
&= \sum_{j=1}^{\infty}p_j(t)\int_{s=0}^{\infty}\mathbb{P}
\left( \left[\sum_{i=1}^{j}X_i+t\right]^+>s\right)ds \\
&=
\sum_{j=1}^{\infty}p_j(t)\int_{s=0}^{\infty}\sum_{k=0}^{j-1}\frac{1}{k!}e^{-\mu
(s-t)}(\mu(s-t))^{k}ds.
\end{align*}

\end{proof}
\begin{remark}
The cost functions for $t<0$ and $t>0$ differ in two aspects. The first one is that an individual who arrives at $t<0$ needs to consider all three possible types of costs. However, an individual who arrives at $t>0$ avoids any earliness cost, and hence must consider only two types of costs. The second difference relates to the tardiness cost. For those who arrived after time 0, the tardiness time is simply the expected waiting time plus the tardiness time upon arrival. This is opposed to the case for those who arrive at $t<0$, where we need to consider the fact that there may not be a tardiness cost. \label{remark:costs}
\end{remark}

\subsection{The fluid model}
In the fluid model each customer is considered as a tiny drop, where
each drop is served in zero time. As opposed to the regular
stochastic model, here we assume a constant volume of drops, which
is denoted by $\Lambda$. The server serves $\mu$ of them in each
time unit. If the total volume in the queue upon arrival is $\eta$,
then the waiting time will be exactly $\frac{\eta}{\mu}$. The
earliness, tardiness and waiting parameters, $\beta_1$, $\beta_2$
and $\alpha$, are as above. We denote by $F(t)$ the proportion of
drops that arrived prior to time $t$. Fluid models are less
realistic but can occasionally provide a good approximation to the
stochastic model. They fit better in cases where both the arrival
volume and the service rate are high, and no randomness in service
is involved. \replaced{They}{It} allows us to give an explicit and
simple formula for the Nash equilibrium and the socially optimal
strategies, which makes computing the PoA possible.
\section{The case without opening and closing time constraints}\label{sec:poisonnocons}
In this section we present our analysis of the unconstrained models, i.e., $T_1=\infty$, and $T_2=\infty$, for both the stochastic and the fluid models.
\subsection{Stochastic model} \label{sec:poisonnoconssub}
In this section we obtain the symmetric Nash equilibrium arrival strategy. We show that it is a continuous distribution with a support that is a connected interval, $(-T_{e_1},T_{e_2})$, with  $T_{e_1},T_{e_2}\geq 0$. We observe that $f(t)$, the density function associated with the Nash equilibrium strategy, is continuous everywhere except at the opening point $-T_{e_2}$, and at time 0.  When we study the properties of the distribution that represents a symmetric Nash equilibrium, we observe that it cannot contain an atom. This is because, if there were an atom and a customer arrived an infinitesimally small moment earlier, he would achieve an atomic gain by paying an infinitesimal cost. We prove the above-mentioned properties in Lemmas~\ref{lem:continoius density} and~\ref{lem:discontin}.
\begin{lem}
In the unconstrained model, the mixed equilibrium strategy comes with an almost everywhere continuous density function, and has connected support (that is, there are no gaps in the arrival distribution).
    \label{lem:continoius density}
\end{lem}
\begin{proof}
Our proof is based on a contradiction argument. That is, we assume the existence of a gap and then show that the individual benefits from deviating, and prefers not to adopt the same strategy that all the others do. We treat two mutually exclusive cases separately. These two cases are 1) the gap interval is either to the left of time 0 or contains time 0, and 2) the gap interval is to the right of time 0. The proof for case 2 (the entire gap is positive) appears precisely in~\cite{Haviv13}. Thus, we provide here a proof for the first case only. Let $\tau_0$ and $\tau_1$ be the lower and upper bounds of the gap, respectively, with $\tau_0<0$. We now examine what the individual can gain if he postpones his arrival by a small $\Delta$. By doing so, he reduces his earliness cost by $\Delta \beta_1$, and his waiting cost by $\alpha(1-p_0(\tau_0))\Delta+o(\Delta)$, and as long as he arrives before time 0 his tardiness cost remains the same. Thus, we eliminate gaps with a lower bound that is smaller than 0. This means that case 1 is not possible.
\end{proof}
\noindent We observe that, for any arrival time within the equilibrium
support, the expected cost is $\beta_1 T_{e_1}$. This is because, if
the first customer to arrive comes at time $-T_{e_1}$, he will
clearly pay only an earliness cost. Hence, his cost is $\beta_1
T_{e_1}$ and therefore his expected cost \replaced{too}{is as well}.
Combined with the fact that the cost function is constant for every
value of $t$ along $(-T_{e_1},T_{e_2})$, this will be the expected
cost for all customers at equilibrium regardless of when they
arrived. So far, we have given general properties of the arrival
distribution, $f(t)$. However, $f(t)$ is yet to be obtained. Since
an analytic expression for obtaining $f(t)$ is implicit, we focus on
how to compute it numerically. We next present a set of conditions
which are satisfied uniquely by $T_{e_1}$, $T_{e_2}$ and $f(t)$.
Then, we describe a numerical procedure that computes them. First,
due to the lack of atoms,
\begin{align}
   \int_{t=-T_{e_1}}^{T_{e_2}} f(t)dt=1. \label{eq:f(t)=1)}
\end{align}
Second, due to the fact that  $C(t)$ is constant along $[-T_{e_1},T_{e_2}]$, and hence the derivative is zero, we have
\begin{align}
- \beta_1 +\alpha w'(t) + \beta_2 \left( \mathbb{E}
\left(\sum_{i=1}^{N(t)}X_i+t\right)^{+}\right)^{'} =0, \quad
-T_{e_1}\leq t < 0, \label{eq:derivative1=0}
\end{align}
\begin{align}
\hspace{3.5cm}(\alpha+\beta_2)w'(t) +\beta_2 =0, \quad 0 < t \leq
T_{e_2}. \label{eq:derivative2=0}
\end{align}
By standard arguments from dynamics in queueing processes we get that
\begin{align}
w'(t)=\frac{\l f(t) -\mu (1-p_0(t))}{\mu}, \quad 0 \leq t \leq
T_{e_2}. \label{eq:dynamic}
\end{align}
Similar to~\cite{Haviv13}, by combining (\ref{eq:derivative2=0}) and (\ref{eq:dynamic}), we conclude that
\begin{align}
f(t)= \frac{(1-p_0(t))\mu}{\l}-\frac{\beta_2
\mu}{(\alpha+\beta_2)\l}, \quad 0 \leq t \leq T_{e_2}.\label{eq:ft}
\end{align}
By the forward Kolomogorov equations, we also have
\begin{align}
 p_0'(t)=p_1(t)\mu -p_0(t)\l f(t),     \quad -T_{e_1} \leq t \leq T_{e_2}.        \label{eq:differen0}
\end{align}
and for $k\geq 1$
\begin{align}
p_k'(t)=-p_k(t)(\mu+\l f(t)) + p_{k-1}(t)\l f(t)+p_{k+1}(t)\mu,
\quad -T_{e_1} \leq t \leq T_{e_2}. \label{eq:differenk}
\end{align}
We next identify the behavior of the system at the latest arrival
time (under Nash equilibrium). From~\cite{Haviv13} (Lemma 2.2), we
also have that in the unconstrained model:
 \begin{align}
 f(T_{e_2})=0 \quad \textrm{or, equivalently,} \quad p_{0}(T_{e_2})=\frac{\alpha}{\alpha+\beta_2}. \label{eq:f(Te)=0}
\end{align}
This result will be useful in the numeric procedure.

\begin{remark}\label{remark:numerical1}
After deriving all the conditions above we can use a numerical
procedure to find the equilibrium strategy. For practical reasons,
we limit ourselves to a maximum number of possible arrivals, denoted
by $N_{\max}$, such that $$N_{\max}=\sup\{n\in \mathbb{N}
:\mathbb{P}(N>n)\leq 10^{-2}\}.$$ The
corresponding forward Kolomogorov equation is
\begin{align}
p_{N_{\max}}'(t)=-p_{N_{\max}}(t)\mu + p_{N_{\max}-1}(t)\l f(t),
\quad \quad -T_{e_1} \leq t \leq T_{e_2}. \label{eq:differenn}
\end{align}
At equilibrium, customers are starting to arrive at time $-T_{e_1}$, and therefore the initial conditions will be
\begin{align*}
p_0(-T_{e_1})=1,   p_k(-T_{e_1})=0, \quad \forall k \in\{1,2,...,N_{\max}\}.
\end{align*}
Before we give the guidelines that describe the numerical procedure,
we note that we don't have a closed-form expression for $f(t)$ where
$t<0$. Therefore, we extract it from the
Eqs.~(\ref{eq:derivative1=0}),~(\ref{eq:differen0}),~(\ref{eq:differenk})
and~(\ref{eq:differenn}). The numerical procedure is as follows. We
first guess a value for $-T_{e_1}$. Then, while using the dynamics
in~(\ref{eq:differen0}),~(\ref{eq:differenk})
and~(\ref{eq:differenn}), coupled with the equation we extracted for
$f(t)$ for $t<0$,  we compute $f(t)$ and $p_k(t)$ for any $k\in
\{0,1,...,N_{\max}\}$ from $t=-T_{e_1}$ until $t=0$. Then we do the
same for $t>0$, only now for computing $f(t)$ we use
Eq.~(\ref{eq:ft}) instead. Stop when one of the
conditions~(\ref{eq:f(t)=1)}) or~(\ref{eq:f(Te)=0}) are met. If the
stopping criterion was based on condition~(\ref{eq:f(t)=1)}) then we
guess a smaller value of $T_{e_1}$, otherwise we guess a larger one.
This goes on until
conditions~(\ref{eq:f(t)=1)})~and~(\ref{eq:f(Te)=0}) are met
simultaneously, while allowing an error of $\epsilon <10^{-2}$, that
is, the convergence error of the integral of $f(t)$ is no larger
than $\epsilon$ (see Fig.~\ref{fig:firstgrpagh2}). Finally, we note
that a bisection procedure is required for finding $-T_{e_1}$.
\end{remark}
\begin{lem}
In the unconstrained model, the equilibrium arrival strategy has a
positive density \added{function}, $f(t)$, which is a continuous
function except at the points $t=-T_{e_1}$ and $t=0$.
\label{lem:discontin}
\end{lem}
\begin{proof}
We commence with the point $t=-T_{e_1}$. According to the
equilibrium arrival strategy, customers are starting to arrive at
$t=-T_{e_1}$, which means that $f(-T_{e_1}^-)=0$. We next show that
$f(-T_{e_1}^+)\neq0$, and by that we prove discontinuity at
$-T_{e_1}$. Due to the fact that the system is empty until time
$-T_{e_1}$, using~(\ref{eq:differen0}) we have
$f(-T_{e_1}^+)=-\frac{p'_0(-T_{e_1}^+)}{\l}$. Since customers are
starting to arrive at time $-T_{e_1}$, the probability of zero
customers in the system is decreasing and hence
$p'_0(-T_{e_1}^+)<0$; therefore $f(-T_{e_1}^+)>0$ as well. The next
discontinuity point is at $t=0$. From~(\ref{eq:dynamic}) one can see
that if $w'(t)$ has a discontinuous point at $t=0$, then $f(t)$ will
have the same. This is because only $f(t)$ can be discontinuous in
the equation. From~(\ref{eq:derivative1=0}) we get $ w'(t)
=\frac{\beta_1-\beta_2 \left(\mathbb{E}
\left(\sum\limits_{i=1}^{N(t)}X_i+t\right)^{+}\right)^{'}}{\alpha}
$, and specifically for $t=0^{-}$ we assert that $\lim_{t\rightarrow
0} \mathbb{E} \left(\sum\limits_{i=1}^{N(t)}X_i+t\right)^{+}=w(0^-)$, and hence
$$ w'(0^{-}) =\frac{\beta_1}{\alpha+\beta_2}, $$ and
from~(\ref{eq:derivative2=0}) we get
$w'(t)=\frac{-\beta_2}{\beta_2+\alpha}$ and specifically, for
$t=0^+$:$$ w'(0^{+}) =\frac{-\beta_2}{\alpha+\beta_2}. $$ Thus, $
w'(0^{-})\neq w'(0^{+})$, which according to~(\ref{eq:dynamic})
means that $f(t)$ has a discontinuity at $t=0$.
\end{proof}

\begin{thm}
In the unconstrained model, a unique equilibrium arrival strategy
with a positive density $f(t)$ along the interval $(-
T_{e_1},T_{e_2})$ will be formed. $f(t)$, $T_{e_1}$ and $T_{e_2}$
obey Eqs.~(\ref{eq:f(t)=1)}) to~(\ref{eq:f(Te)=0}), where
$T_{e_1}>0$ and $T_{e_2}>0$.  $f(t)$ is continuous except at $t=0$.
\label{thm:f(t)unique}
\end{thm}
Before proving the theorem we provide the following lemma and the
definitions. Let there be two different arrival strategies with the
cumulative arrival distribution $F_1(\cdot)$ and $F_2(\cdot)$ and
let $w_1$ and $w_2$ be the lower bounds of the arrival support of
$F_1$ and $F_2$ respectively, where $w_1>w_2$.  Also, let
$N_1(\cdot)$ and $N_2(\cdot)$ be the number of customers in the
system according to $F_1$ and $F_2$ respectively. Finally, let $t_0$
be the lowest value in which $F_1(\cdot)$ and $F_2(\cdot)$ meet.
Formally, $t_0=\inf\{t|F_1(t)=F_2(t),t>w\}.$

\begin{lem}
The expected cost of arriving at $t_0$ under $F_2$ is lower than
arriving under $F_1$.\label{lem:coupling}
\end{lem}

\begin{proof}
We first state that the difference of the two expected costs can be
only via the workload since they arrived at the same time. By
employing a coupling argument, we claim that $N_2(t_0)\leq _{st}
N_1(t_0)$. This is because the arrivals prior to $t_0$ according to
$F_2$ occurred earlier than those according to $F_1$. Hence, according
to $F_2$, service completions occurred no later than those according
to $F_1$.\ Consequently the workload under $F_2$ is
smaller and by that we complete the proof.
\end{proof}
\noindent We next prove Theorem~\ref{thm:f(t)unique}.
\begin{proof}

We need to prove that the equilibrium is unique, since everything
else has already been argued. Recall that by setting a starting
arrival point at equilibrium, $-T_{e_1}$, and due to the a constant
cost function (under equilibrium), a unique $f(t)$
 results. Furthermore, $f(t)$ and $-T_{e_1}$ determine the value of
 $T_{e_2}$ uniquely. Aiming for a contradiction we assume an alternative equilibrium strategy in which customers are arriving
 along the support $(-T'_{e_1},T'_{e_2})$ where, $-T'_{e_1}<-T_{e_1}$ (w.l.o.g.).\ We note that the
 expected costs of the original and alternative equilibria are $\beta_1 T_{e_1}$ and $\beta_1 T'_{e_1}$ respectively, with $\beta_1 T_{e_1}<\beta_1 T'_{e_1}$.\ From Lemma~\ref{lem:coupling},
  we get that there exist a point $t_0$ in which the expected cost
  under the alternative strategy is lower than arriving under the original
  strategy since $-T_{e_1}'<-T_{e_1}$. Since under mixed equilibrium the cost is constant, the fact that the orders between the costs at opening and at $t_0$ are reversed, rules out the existence of two equilibria.

\end{proof}
\subsection{Fluid model}
Let's consider the first drop to arrive. Clearly, it will not be after time $0$. Moreover, this drop will be served immediately. This means that only earliness and waiting costs are under consideration. As time goes by and a workload accumulates, there will be a point in time at which a drop that arrived prior to time $0$ will enter service after time $0$. Hence, all types of cost are under consideration. Once drops start to arrive after time $0$, they avoid earliness cost. Consequently, there are three different time intervals with different cost function structures, and hence the mixed equilibrium arrival strategy includes three different formulas for the density. Our notation for $T_{e_1}$ and $T_{e_2}$ remains the same as for the stochastic model.
\begin{thm}\label{thm:fluidnocons}
In the unconstrained model, the unique equilibrium density is a
stepwise function\deleted{s} with three steps.
\begin{align*}
  f(t) = \begin{cases}
\frac{\alpha+\beta_1}{\alpha}\frac{\mu}{\Lambda},  &\hspace{0.85cm}  -\frac{\beta_2 \Lambda}{(\beta_1+\beta_2)\mu}\leq t\leq-\frac{\beta_1}{\alpha+\beta_1}\frac{\beta_2 \Lambda}{(\beta_1+\beta_2)\mu},\\
\frac{\alpha+\beta_1}{\alpha+\beta_2}\frac{\mu}{\Lambda},     &  -\frac{\beta_1}{\alpha+\beta_1}\frac{\beta_2 \Lambda}{(\beta_1+\beta_2)\mu}\leq  t \leq 0,   \\
\frac{\alpha}{\alpha+\beta_2}\frac{\mu}{\Lambda},       &
\hspace{2.3cm} 0\leq  t \leq \frac{\beta_1
\Lambda}{(\beta_1+\beta_2)\mu}.
                          \end{cases}
\end{align*}
 The corresponding social cost is
 \begin{align*}
 \frac{\beta_1\beta_2 \Lambda^2}{(\beta_1+\beta_2)\mu}.
 \end{align*}
 \end{thm}
\begin{proof}
Since we seek for a symmetric Nash equilibrium, the cost function is
constant along the support of the arrival distribution, and hence the derivative is zero. In the first part,
only earliness and waiting costs are considered until a point in
time at which arrivals are being served after time $0$. In the
second part, all three costs are taken into account. This means that
a drop that arrives at the queue at $t<0$ will exit at $t>0$. The
third part contains both waiting and tardiness costs. In summary, we
have three different parts with three cost functions, and hence
three different arrival densities. Let $[-T_{e_1},t_1]$ be the first
part's arrival density period, where $t_1<0$ is the time point after
which arrivals will pay a tardiness cost. The cost function of this
part is\deleted{:} $$ C_1 (t)=\frac{\Lambda
F(t)-\mu(t+T_{e_1})}{\mu}\alpha-\beta_1 t, \quad t\in [-T_{e_1},
t_1].$$ After taking the derivative and comparing to zero we will
get the result as stated above. We can reason in the same spirit
about the second and third parts, where the cost functions are
$$C_2(t)=\frac{\Lambda F(t)-\mu(t+T_{e_1})}{\mu}\alpha-\beta_1
t+\beta_2\left(\frac{\Lambda F(t)-\mu(t+T_{e_1})}{\mu}+t\right),
\quad t\in [t_1,0]. $$ and $$C_3(t)=\beta_2
t+(\beta_2+\alpha)\left(\frac{\Lambda
F(t)-\mu(t+T_{e_1})}{\mu}\right), \quad t\in [0, T_{e_2}] $$
respectively. The values of $T_{e_1}$, $t_1$ and $T_{e_2}$ are the
solution of the following set of equations: $
\beta_1T_{e_1}=\beta_2T_{e_2}$, $\frac{\Lambda
F(t_1)-\mu(t_1+T_{e_1})}{\mu}=-t_1$ and
$(t_1+T_{e_1})(\frac{\alpha+\beta_1}{\alpha}\frac{\mu}{\Lambda})+\frac{T_{e_1}\beta_1}{\alpha+\beta_1}
(\frac{\alpha+\beta_1}{\alpha+\beta_2}\frac{\mu}{\Lambda})+T_{e_2}\frac{\alpha}{\alpha+\beta_2}\frac{\mu}{\Lambda}=1$.
The first equation arises due to the constant cost for all arriving
drops. Therefore, the first drop to arrive and the last will have
the same cost. The second equation shows that a drop arriving at
$t_1$ will exit the system at time 0. Finally, the third equation
sums the pdf to 1. The corresponding social cost is simply $\Lambda$
times the overall cost of each drop.
\end{proof}
We next present the socially optimal strategy and its overall cost.
\begin{thm}
The unique socially optimal strategy has a uniform density of $\frac{\mu}{\Lambda}$ along $[-\frac{\beta_2\Lambda}{(\beta_1+\beta_2)\mu},\frac{\beta_1 \Lambda}{(\beta_1+\beta_2)\mu}]$. The corresponding social cost is
\begin{align}
\frac{\beta_1\beta_2
\Lambda^2}{(\beta_1+\beta_2)2\mu}.\label{eq:noconsoptiamlsocialcost}
\end{align}
\end{thm}
\begin{proof}
The socially optimal arrival rate is uniform with rate $\mu$ in order to avoid waiting cost. The values of the lower and upper bounds of the arrival (which we denote by $L\!B$ and $U\!B$ respectively) are the ones that minimize the social welfare cost function, where clearly\\ $L\!B,U\!B \geq 0$:
  \begin{align}
  -\int_{-L\!B}^{0}\beta_1\mu tdt+\int_{0}^{U\!B}\beta_2\mu tdt. \label{eq:optiamlcost}
  \end{align}
In addition, in order to ensure that the total volume is $\Lambda$, the condition $(L\!B+U\!B)\mu=\Lambda$ should be met. Now obtaining the values of $L\!B$ and $U\!B$ is straightforward. Finally, the social welfare results directly from~(\ref{eq:optiamlcost}).
\end{proof}
\noindent The PoA is a direct result of dividing~(\ref{eq:noconsoptiamlsocialcost}) by~(\ref{eq:optiamlcost}). In the unconstrained model it equals exactly 2.
\begin{remark}
In the arrival density at equilibrium the first step is larger than the second step, and the third is smaller than the second.
\end{remark}
\begin{figure}[h]
    \begin{center}
        \includegraphics[height=8cm, width=13cm]{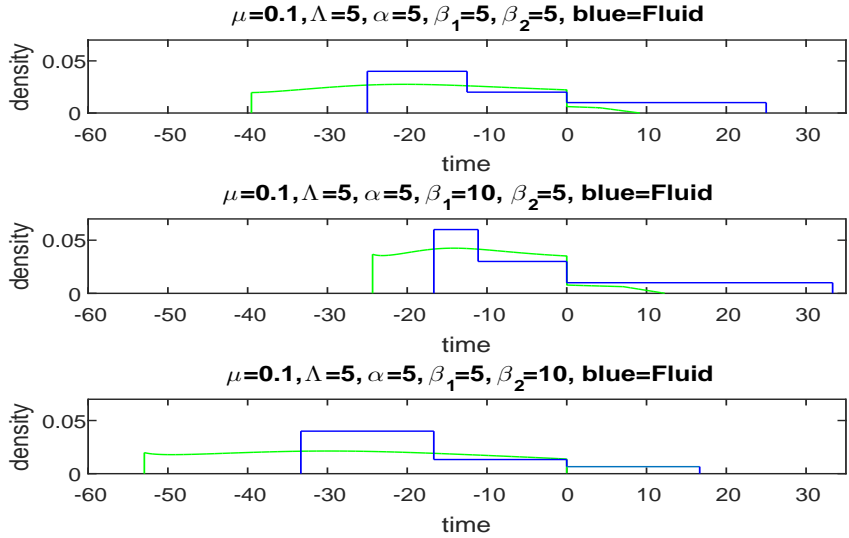}\caption{ Numerical example: the equilibrium arrival distribution in the unconstrained model, for three different set values of $\beta_1$, $\beta_2$ and $\alpha$.}\label{fig:firstgrpagh2}
    \end{center}
\end{figure}
\section{The case with opening and closing time constraints}\label{sec:poisoncons}
In this section we present our analysis of the constrained models, i.e., where $T_1<\infty$ and $T_2<\infty$, for both the stochastic and the fluid models.
\subsection{Stochastic model}
In this subsection we study the effect of arrival time constraints on the equilibrium arrival strategy.  Due to the complexity of the case where both of the constraints are effecting customers behaviour, we first discuss the case where only the closing time is constrained (i.e., $T_1=\infty$, $T_2<\infty$), then the case where only the opening time is constrained (i.e., $T_1<\infty$, $T_2=\infty$). Finally, consider both of them simultaneously.
\subsubsection{Constrained closing time:}

  While studying the effect of constraining the closing time, we first observe that if the unconstrained equilibrium meets the constraint (i.e., $T_2>T_{e_2}$), then the constraint has no effect on the customers' behaviour. Thus, we deal here only with the opposite case, $T_2<T_{e_2}$. As before, the equilibrium arrival strategy is a mixed strategy along an interval $(-T'_{e_1},T_2)$
  with a complete density $f(t)$ (that is, without any atoms). The lower bound of the support of $f(\cdot)$, $T'_{e_1}$, is a part of the solution and needs to be found. We follow the same line of analysis as in Theorem~\ref{thm:f(t)unique}, but with two adjustments. The first one is that $T_{e_2}$ from Eq.~(\ref{eq:f(t)=1)}) is replaced by $T_2$ and the second is that Eq.~(\ref{eq:f(Te)=0}) no longer holds. Actually, it is no longer in use because the closing time is not exogenous. Notice that if $T_1$ is finite but not effective (i.e., $T'_{e_1}<T_1$), then our results remain the same. We give further details and prove our results in Theorem~\ref{thm:constraints}.
\subsubsection{Constrained opening time:}
While studying the effect of constraining the opening time, we first observe, just like the case above, that if the unconstrained equilibrium meets the constraint (i.e., $T_1>T_{e_1}$), then the constraint has no effect on the customers' behaviour. Thus, we deal here only with the opposite case, $T_1<T_{e_1}$. The structure of the equilibrium distribution in this case is significantly different to those we have seen before. The difference is that a complete density starting at time $-T_1$ cannot be an equilibrium arrival strategy. Instead, the structure of the equilibrium is either one in which all customers arrive at time $-T_1$, or else there is an atom of size $p<1$ at time $-T_1$, followed by a gap, and then a positive density.
For further analysis, we give the following definitions. Let $N_p$ be a random variable that follows a Poisson distribution with parameter $\l p$. Let $g_1(t,p)$ be the expected overall cost of arriving at time $t>-T_1$, given that $N_p$ arrives at $-T_1$ and no-one else arrives in the interval $(-T_1,t)$. Let $g_2(p)$, be the expected overall cost of arriving at $-T_1$ among $N_p$ more customers.
If everyone arrives at $-T_1$, arriving a moment later is clearly not beneficial. Moreover, arriving at any time $t<0$ is suboptimal, since the longer a customer waits, approaching time 0, his expected earliness and waiting costs both reduce, while his tardiness cost remains the same. Therefore, the best response to everyone else arriving at $-T_1$ is either arriving at $-T_1$ as well, or at some $t\geq 0$. Let $t^*$ be the best response, besides $-T_1$; then $t^* = \arg\min_{t\geq 0}\thinspace g_1(t,1)$. If $g_1(t^*,1)>g_2(1)$ then there is a pure equilibrium at time $-T_1$. That is, even the optimal arrival point besides arriving with everyone else at $-T_1$ is more expensive than simply arriving at $-T_1$ as well. However, if $g_1(t^*,1)<g_2(1)$, the equilibrium arrival strategy is mixed, and incorporates an atom at time $-T_1$, followed by a gap, and then a positive density. Now, if arriving at $t^*$ is better than arriving with everyone else at $-T_1$ it can no longer hold for symmetric Nash equilibrium. Consequently, an atom which is smaller than 1 is created at $-T_1$, and only after a while customers will arrive again. That is, there exists an atom size $p'$ such that $g_1(t,p')=g_2(p')$ for some $t>0$. We denote the upper bound of the arrival support as $T'_{e_2}$ and distinguish it from the notation of the upper bound when there was no effective opening time constraint $T_{e_2}$. We give further details and prove our results in Theorem~\ref{thm:constraints}.
\subsubsection{Both opening and closing times are constrained:}
To this end, we have discussed cases in which only one of the time
constraints is effective. We next address the cases where both of
them are. We split into two cases in order to simplify. First, let
us consider the case where $T_2<T_{e_2}$ and $T_1<T'_{e_1}$. One can
look at this scenario as if only the closing time is constrained
(while ignoring the opening time constraint), and when computing the
arrival strategy it resulted in a new initial arrival point
$T'_{e_1}$. However, if $T_1<T'_{e_1}$ the solution in no longer
valid. Clearly, the equilibrium arrival strategy can no longer
involve a complete density at time $-T'_{e_1}$. In fact, an atom at
time $-T_1$ is guaranteed. Furthermore, a pure equilibrium strategy
is possible. Now, let us explore a second case in which
$T_1<T_{e_1}$ and we relate to $T_2$ later. That is, we ignore the
closing time constraint for the time being. If $g_1(t^*,1)>g_2(1)$,
then a pure strategy is formed regardless to $T_2$. If
$g_1(t^*,1)<g_2(1)$, we define the following: Let $T'$ be the
smallest $t$ that satisfies $g_1(t,1)=g_2(1)$. Suppose now that
$T'>T_2$, this means that a pure strategy is formed again (where
now, $T_2$ is effective). But, if $T'<T_2$, then, this means that
there is an arrival time that is better than arriving at $-T_1$ with
everyone else. In this case, customers will start arriving from $T'$
until $T'_{e_2}$ if there was no effective closing time constrained.
However, if $T_2<T'_{e_2}$, then there will be an atom which is
smaller than 1 at $-T_1$ and then customers arrive from $T'$ until
$T_2$. We summarize our results for all of the various cases in the
following theorem. We present them by dividing the possible cases
into three different types of equilibrium arrival strategies: 1)
pure, 2) mixed with an atom, and 3) mixed without atoms.
\begin{thm}\label{thm:constraints}
One of the following exhaustive and mutually exclusive
cases occurs.\\
 Case 1:  The equilibrium strategy is pure and it prescribes arriving at $-T_1$ if $g_2(1)<g_1(t^*,1)$ or $T_2<T'$ and $g_1(t^*,1)<g_2(1)$.\\
Case 2: A mixed equilibrium which incorporates an atom with size $p'$ at $t\!=\!-T_1$, then an interval $(-T_1,t')$ with zero density, followed by a positive density from $t=t'$ until $(T'_{e_2}\wedge T_2)$. This will be formed if one of the following two conditions occurs:\\
 \textbf{\romannum{1 }}.\ $T_{e_1}<T_1<T'_{e_1}$ and  $T'<T_2<T_{e_2}$, or \\
  \textbf{\romannum{2 }}.\ $T_1<T_{e_1}$ and $T_2>T'$. The following conditions need to be obeyed.
 \begin{align}
 g_1(t',p')=g_2(p'). \label{eq:g1(t',p)}
 \end{align}
  \begin{align}
   \int_{t=t'}^{(T'_{e_2}\wedge T_2)}f(t)dt=1-p'. \label{eq:f(t)=1-p}
 \end{align}
 \begin{align}
f(t)= \frac{(1-p_0(t))\mu}{\l}-\frac{\beta_2
\mu}{(\alpha+\beta_2)\l}, \quad t' \leq t \leq (T'_{e_2}\wedge
T_2).\label{eq:ft1}
\end{align}
 The initial conditions at time $t'$ are
  \begin{align}
  p_k(t')=\sum_{n=k}^{\infty}\frac{e^{-\l p} (\l p)^n}{n!}\frac{e^{-\mu(t'+T_1)}(\mu(t'+T_1)^{n-k})}{(n-k!)},\quad k\geq 1.\label{eq:initialpois}
  \end{align}
 \begin{align}
 p_0(t')=1-\sum_{k=1}^{\infty}p_k(t') . \label{eq:initialcons0}
  \end{align}
 and
\begin{align}
f(t')= \frac{(1-p_0(t'))\mu}{\l}-\frac{\beta_2
\mu}{(\alpha+\beta_2)\l}\geq 0, \quad t' \leq t \leq (T'_{e_2}\wedge
T_2).\label{eq:ft1initial}
\end{align}
Finally, if $T'_{e_2}<T_2$ then
\begin{align}
f(T'_{e_2})=0. \label{eq:f(T'e_2=0)}
\end{align}
Case 3: A mixed equilibrium strategy without atoms, and with a discontinuity point at time 0. This will be formed if one of the following occurs:\\
 \textbf{\romannum{1 }}.\ $T_1>T_{e_1}$ and $T_2>T_{e_2}$, or \\
 \textbf{\romannum{2 }}.\ $T'_{e_1}<T_1$ and $T_2<T_{e_2}$.\\ \textbf{\romannum{1 }} is equivalent to the case in Theorem~\ref{thm:f(t)unique}. \textbf{\romannum{2 }} is equivalent to the case in Theorem~\ref{thm:f(t)unique}, with two exceptions. The first one is the support of $f(t)$, which is now along $(-T'_{e_1},T_2]$. The second is that condition~(\ref{eq:f(Te)=0}) is no longer satisfied.
\end{thm}
\begin{proof}
Case 1: The condition $g_1(t^*,1) > g_2(1)$ implies that the cost associated with the best response to every other player arriving at time $-T_1$ is worse than just arriving at $-T_1$. $T_2<T'$, under the assumption that $g_1(t^*,1)<g_2(1)$, implies that all arrival times that have lower overall cost than arriving with everyone else at time $-T_1$ are no longer possible within the system constraints.\\
Case 2: Since $T_2>T'$ and $g_1(t^*,1)<g_2(1)$, there is no pure equilibrium. In addition, a density that starts at time $-T_1$ cannot be part of the equilibrium. The reason for that is as follows. If there were a density starting a time $-T_1$ at equilibrium, then a customer who arrives at time $-T_1$ suffers only the earliness cost, and hence his total cost is $\beta T_1$. As, under mixed equilibrium, the customers are indifferent about when to arrive, this is the total cost of
everyone else. It is equivalent to either $T'_{e_1}=T_1$ (under condition\emph{ \textbf{\romannum{1 }}}) or to $T_{e_1}=T_1$ (under condition \emph{\textbf{\romannum{2 }}}). Note that neither of the latter two conditions is the case. In addition, once an atom is created at $-T_1$, arriving a moment later is strictly suboptimal and hence a gap in the
interval $[-T_1,t']$ is created. Eq.~(\ref{eq:g1(t',p)}) is obeyed, due to the fact that the expected total cost is constant along the support. Eqs.~(\ref{eq:f(t)=1-p}) and~(\ref{eq:ft1}) are obeyed by the same consideration of Eqs.~(\ref{eq:f(t)=1)}) and~(\ref{eq:ft}). Eqs.~(\ref{eq:initialpois}) and~(\ref{eq:initialcons0}) follow from simple probability considerations. Eq.~(\ref{eq:ft1initial}) implies that customers are starting to arrive for the first time after time $-T_1$ at $t'$, which is clearly the case. Finally, we assert that Eq.~(\ref{eq:f(T'e_2=0)}) holds, by using the same argument starting from~(\ref{eq:f(Te)=0}).\\
Case 3: Condition \emph{\textbf{\romannum{1 }}} is equivalent to Theorem~\ref{thm:f(t)unique}, since the opening and closing time constraints are not effective. Condition \emph{\textbf{\romannum{2 }} }suggests that only the closing time constraint is effective. This means that no atoms can be formed as argued above in subsection~\ref{sec:poisonnoconssub}.
\end{proof}
We next provide the steps needed for numerical computation of the mixed equilibrium in the constrained stochastic model.
\begin{remark}
   If the equilibrium arrival strategy is mixed with an atom, we make a guess for $p'$, then compute $t'$ such that $g_1(t',p')=g_2(p')$, where
   \begin{align*}
   g_1(t,p) =  \begin{cases}
              -\beta_1 t+\alpha \mathbb{E}\left(\left[\sum\limits_{j=1}^{N_p} X_j\!-\!T_1\!-\!t\right]^+\right) +\beta_2 \mathbb{E}\left(\left[\sum\limits_{j=1}^{N_p} X_j-T_1\right]^+\right),\quad  &  t<0,\\
             \beta_2 t+(\alpha+\beta_2)\mathbb{E}\left(\left[\sum\limits_{j=1}^{N_p} X_j-T_1-t\right]^+\right),\quad  &   t>0. \end{cases}
\end{align*}
\begin{align*}
g_2(p)=\beta_1T_1+\alpha\frac{\l
p}{2\m}+\beta_2\left(\sum\limits_{i=1}^\infty \frac{e^{-\l p} (\l
p)^i}{i!}\sum\limits_{k=1}^{i}\frac{1}{i}\mathbb{E}\left[\left(\sum\limits_{j=1}^{k}X_j-T_1\right)^+\right]\right).
\end{align*}
   Given $t'$, we compute the initial conditions by~(\ref{eq:initialpois}) and~(\ref{eq:initialcons0}). After obtaining $t'$ and the initial conditions at time $t'$, we use a similar procedure to the one described in Remark~\ref{remark:numerical1}. We first compute $T'_{e_2}$ and ignore for the time being the fact that it may be larger than $T_2$. This is done by computing $f(t)$ and $p_k(t)$ for any $k\in \{0,1,...,N_{\max}\}$ from $t'$ until one of the conditions~(\ref{eq:f(t)=1-p}) or~(\ref{eq:f(T'e_2=0)}) is met. The probabilities $p_0(t), p_1(t),...,p_{N_{\max}(t)}$ are computed using
   \begin{align}
p_0'(t)=p_1(t)\mu -p_0(t)\l f(t), \quad      \quad t' \leq t \leq
(T'_{e_2}\wedge T_2).  \label{eq:differen01}
\end{align}
 \begin{align}
p_k'(t)\!=\!-p_k(t)(\mu\!+\!\l f(t))\!+\! p_{k-1}(t)\l f(t)\!+\!
p_{k+1}(t)\mu,   \quad t' \leq t \leq (T'_{e_2}\wedge T_2).
\label{eq:differenk1}
\end{align}
      If the stopping criterion was due to condition~(\ref{eq:f(t)=1-p}), then we guess a smaller value of $p'$; otherwise we guess a larger one. This goes on until conditions~(\ref{eq:f(t)=1-p}) and~(\ref{eq:f(T'e_2=0)}) are met simultaneously. After obtaining $T'_{e_2}$, we can check if $T'_{e_2}<T_2$ or not. If $T'_{e_2}<T_2$, then we finished. If $T'_{e_2}>T_2$, we again make a guess for $p'$, although now we compute $f(t)$ until time $T_2$, and we repeat this until condition~(\ref{eq:f(t)=1-p}) is met at $T_2$. Finally, if the equilibrium arrival strategy is mixed without atoms, we repeat the same procedure as in Remark~\ref{remark:numerical1}, with $T_{e_2}$ replaced by a fixed $T_2$. As in the previous section, this is achieved by allowing an error of $\epsilon <10^{-2}$ (see Figs~\ref{fig:atomdiffT1} and~\ref{fig:atomdiffmu}).
\end{remark}
\begin{figure}[h]
    \begin{center}
        \includegraphics[height=7cm, width=15cm]{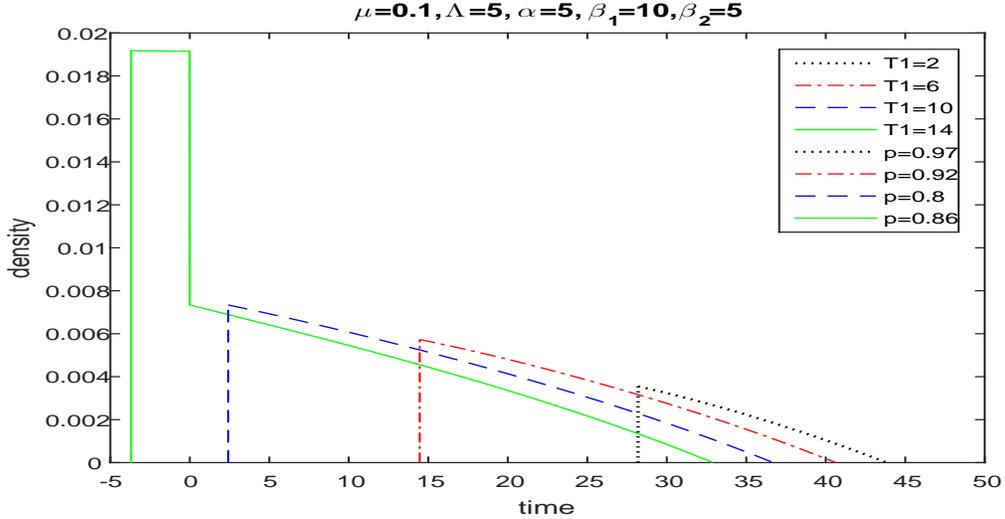}\caption{ Numerical example: the equilibrium arrival distribution in the constrained model for different values of $T_1$. The equilibrium arrival strategy under the green graph, has an atom at $t=-14$ of size $p=0.86$. The gap is until $t=-4.15$. Arrival in the interval $[-4.15,0)$ comes with an earliness cost (among possible waiting and tardiness), but by arriving within the interval $[0,33.2]$, one avoids cost. In the blue, red and black graphs, earliness cost are only for those who arrived in the atom at $-T_1$, whereas, after the gap, customers arrive after $t=0$. }\label{fig:atomdiffT1}
    \end{center}
\end{figure}

\begin{figure}[h]
    \begin{center}
        \includegraphics[height=7cm, width=15cm]{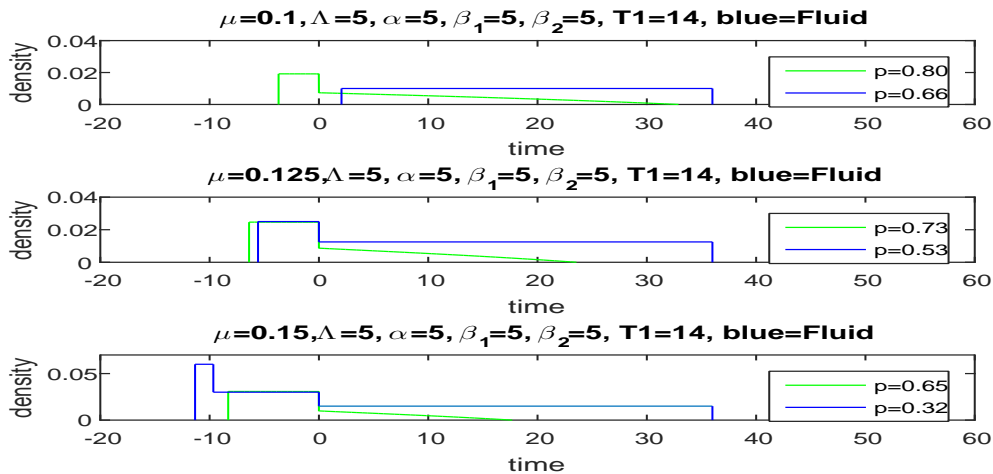}\caption{ Numerical example: the equilibrium arrival distribution in the constrained model for different values of $\mu$. The three grpahs illustrates all types of possible equilibria. For example, in the lowest graph (i.e. $\mu=0.15$), there is an atom at $-14$ of size $p=0.65$ and $p=0.32$ under the stochastic and fluidic model, respectively. Under the stochatic model (the green graph), there are two different probability densities (arriving prior and after 0) for those who arrive after the atom. Under the fluidic model, there are three different probability densities: (1) $[-11.5,-10.1]$, customers have earliness and waiting costs. (2) $(-10.1,0]$, customers have all cost types. (3) $(0,36.9]$, customers have waiting and tardiness costs.     }\label{fig:atomdiffmu}
    \end{center}
\end{figure}

\subsection{Fluid model}
As in the previous section, the total amount of time the server is needed in the fluid model is $\frac{\Lambda}{\mu}$. Under the opening and closing time constraints, the total amount of time during which customers are allowed to arrive is $T_1+T_2$. If it is larger than the total time needed, i.e., $\frac{\Lambda}{\mu}<T_1+T_2$, then at most one of the two constraints is tight. Specifically, if $T_1<T_{e_1}$ then customers are starting to arrive from time $-T_1$ until some $t<T_2$ (possibly pure at time $-T_1$). If $T_2<T_{e_2}$ then customers are starting to arrive from some $t>-T_1$ until time $T_2$. We deal with both cases separately in Theorems~\ref{thm:fluidT1} and~\ref{thm:fluidT2}, respectively. Of course, if both $T_1<T_{e_1}$ and $T_2<T_{e_2}$, then the situation coincides with our results from Theorem~\ref{thm:fluidnocons}. However, if the total amount of time customers are allowed to arrive is smaller than the total time needed to serve them (i.e., $\!T_1\!+\!T_2<\frac{\Lambda}{\mu}$) then things are more involved. This is because both the opening and closing times of the system need to be considered while obtaining the equilibrium arrival strategy. We deal with this case in Theorem~\ref{thm:fluidT1T2}. For each one of the cases above we specify the corresponding socially optimal strategy.
Before presenting the equilibrium arrival strategy, we define several points in time and probabilities that play a role in the specific structure of the equilibrium.
\begin{align*}
&A_1=\frac{(-\beta_1\!+\!\sqrt{\beta_1^2-\alpha\beta_2\!+\!\beta_2^2})\Lambda} {\beta_2\mu}, A_2=\frac{-\beta_1\beta_2\Lambda+\beta_2\Lambda\sqrt{(\alpha\!+\!\beta_1)^2\!+\!\alpha\beta_2\!+\!\beta_2^2}}{(\alpha^2\!+\!\beta_2^2\!+\!\alpha(2\beta_1\!+\!\beta_2))\mu},\\
&A_3=\frac{2\beta_2\Lambda}{(\alpha+2\beta_1+2\beta_2)\mu},\quad A_4=\frac{\beta_2 \Lambda}{(\beta_1+\beta_2)\mu},\\
&p_1=\frac{\beta_2\Lambda^2-T_1\beta_1\mu+\frac{1}{2}\sqrt{-4T_1^2\beta_2(\alpha+\beta_2)\Lambda^2\mu^2+(2\beta_2\Lambda^2-2T_1\beta_1\Lambda\mu)^2}}{(\alpha+\beta_2)\Lambda^2}\\
&p_2=\frac{2(\beta_2\Lambda-T_1\mu(\beta_1+\beta_2))}{\alpha \Lambda},\\
&t_1=\frac{-T_1(\alpha+\beta_1)\Lambda\mu+\sqrt{\Lambda^2(\beta_2\Lambda^2-2T_1\beta_1\beta_2\Lambda\mu+T_1^2(\beta_1^2-\beta_2(\alpha+\beta_2))\mu^2)}}{\alpha\Lambda\mu},\\
&t_2=-T_1+\frac{\sqrt{\Lambda^2(\beta_2^2\Lambda^2-2T_1\beta_1\beta_2\Lambda\mu
+T_1^2(\beta_1^2-\beta_2(\alpha+\beta_2))\mu^2)}}{(\alpha+\beta_1)\Lambda\mu},\\
&t_3=\frac{\beta_2\Lambda-T_1(\alpha+2\beta_1+\beta_2)\mu}{(\alpha+\beta_1)\mu},\quad
t_4=\frac{\beta_2(T_1\mu-\Lambda)}{(\alpha+\beta_1)\mu}.
\end{align*}
\begin{thm} \label{thm:fluidT1}
The equilibrium arrival strategy when $T_1\!+\!T_2\geq  \frac{\Lambda}{\mu}$ and $T_1<T_{e_1}$ is as follows.\\
   Case 1: If $T_1\leq  A_1$ then a pure equilibrium will be formed at time $-T_1$.  The corresponding social cost is $\Lambda(\beta_1T_1+\frac{\alpha \Lambda}{2\mu}+\beta_2\frac{\mu(\frac{\Lambda}{\mu}-T_1)^2}{2\Lambda})$.\\
  Case 2:  If $ A_1< T_1\leq A_2 $ the equilibrium arrival strategy incorporates an atom with size $p_1$ at time $-T_1$, followed by a gap until time $t_1>0$. Then there is a positive density of $\frac{\alpha}{\alpha+\beta_2}\frac{\mu}{\Lambda}$ along $[t_1,\frac{\Lambda}{\mu}-T_1]$. The corresponding social cost is $ \Lambda\beta_2(\frac{\Lambda}{\mu}-T_1)$.\\
  \indent The social costs in cases 3 and 4 are identical to this case.\\
    Case 3: If $ A_2< T_1\leq A_3$ the arrival strategy incorporates an atom with size $p_1$ (the same size as the previous case) at time $-T_1$, followed by a gap until time $t_2<0$. Then there is a positive density of $\frac{\alpha+\beta_1}{\alpha+\beta_2}\frac{\mu}{\Lambda}$ along $[t_2,0]$ and then a positive density of $\frac{\alpha}{\alpha+\beta_2}\frac{\mu}{\Lambda}$ along $[0, \frac{\Lambda}{\mu}-T_1]$.\\
  Case 4:  If $ A_3<T_1\leq A_4$ the arrival strategy incorporates an atom with size $p_2$ at time $-T_1$, followed by a gap until time $t_3<0$. Then there is a positive density of $\frac{\alpha+\beta_1}{\alpha}\frac{\mu}{\Lambda}$ along $[t_3,t_4]$, and then a density of $\frac{\alpha+\beta_1}{\alpha+\beta_2}\frac{\mu}{\Lambda}$ along $[t_4,0]$ and a positive density of $\frac{\alpha}{\alpha+\beta_2}\frac{\mu}{\Lambda}$ along $[0, \frac{\Lambda}{\mu}-T_1]$.
\end{thm}
 \begin{proof}
   Case 1: $T_1\leq A_1$ implies that arriving at time $-T_1$ with everyone else is the best response, and hence the equilibrium is pure. The expected disutility of arriving with everyone else at time $-T_1$ is $\beta_1T_1+\frac{\alpha \Lambda}{2\mu}+\beta_2\frac{\mu(\frac{\Lambda}{\mu}-T_1)^2}{2\Lambda}$, while the best response against that is arriving at $t=\frac{\Lambda}{\mu}-T_1$, where the expected disutility is $\beta_{2}(\frac{\Lambda}{\mu}-T_1)$. In fact, $A_1$ is the threshold such that lower values of $T_1$ make the expected disutility of arriving with everyone else at time $-T_1$ smaller. The social cost is just the product of the total volume $\Lambda$ and the expected disutility of each drop.\\ \\
    \indent In the next three cases the equilibrium is mixed with an atom at $-T_1$. The atom is always followed by a gap and then by one, two or three different constant positive densities. The values of the densities are obtained similarly to Theorem~\ref{thm:fluidnocons}, with only their boundaries changing. We further conclude that an atom of size $p=1$ cannot be formed in any of the cases 2, 3, and 4 due to the following: From case 1, when $T_1\leq A_1$, the atom has the size of 1.  In order to rule out the possibility a pure strategy in the other cases, we show that from $t=-A_1$, the atom (i.e. $p_1$ or $p_2$) is decreasing with $T_1$. That is, from the first moment that $T_1>A_1$ the atom is smaller than 1. In the expression that computes the value of $p_1$, $T_1$ appears three times. The first two comes with a negative parameter. The third appears in the expression: $(2\beta_2 \Lambda^2-2T_1\beta_1\Lambda\mu)^2$, we show that the difference (inside the square) is positive and hence $p_1$ decreases when $T_1$ increases. Since $T_1$ is an effective constraint, $T_1<\frac{\beta_2 \Lambda} {(\beta_1+\beta_2)\mu}$ and as a consequence, after some simple algebra, $(2\beta_2 \Lambda^2-2T_1\beta_1\Lambda\mu)>0$. We follow the same line of analysis to show that $p_2<1$.
        \\

 Case 2:  $T_1>A_1$ implies that a pure strategy is no longer possible, which is a direct result from the previous case. $T_1<A_2$ is equivalent to $t_1>0$ which is the case here. $t_1$ and $p_1$ are obtained by solving the equations: $$\beta_1 T_1+\frac{\alpha p_1\Lambda}{2\mu}+\beta_2\frac{(\Lambda p_1-T_1\mu)^2}{2\mu p_1 \Lambda}=\beta_2(\frac{p_1\Lambda}{\mu}-T_1)+\alpha(\frac{p_1\Lambda}{\mu}-T_1-t_1),$$ and $$p_1+\frac{\mu}{\Lambda}\frac{\alpha}{\alpha+\beta_2}(\frac{\Lambda}{\mu}-T_1-t_1)=1.$$
where the first equation is due to similar disutilities at $-T_1$ and $t_1$ (which is of course true for any possible time in the equilibrium arrival strategy). The second equation sums the density probability. To obtain the social cost, we first state that the last drop will arrive exactly at $t=\frac{\Lambda}{\mu}-T_1$, to an empty system. Hence its total cost is $\beta_2(\frac{\Lambda}{\mu}-T_1)$. Consequently, the social cost is $\Lambda\beta_2(\frac{\Lambda}{\mu}-T_1)$. In fact, this is true for cases 3 and 4 as well. Consequently, in all mixed equilibrium cases (i.e., 2, 3 and 4) the equilibrium social costs are the same.\\
   Case 3:  $T_1>A_2$ is equivalent to $t_2<0$. $A_3$ is the threshold which separates this case from case 4, and will be discussed there. $t_2$ and $p_1$ form the solution of the equations: $$\beta_1 T_1+\frac{\alpha p_1\Lambda}{2\mu}+\beta_2\frac{(\Lambda p_1-T_1\mu)^2}{2\mu p_1 \Lambda}=-\beta_1t_1+\alpha(\frac{p_1\Lambda}{\mu}-T_1-t_1)+\beta_2(\frac{p_1\Lambda}{\mu}-T_1),$$ and $$p_1+\alpha\frac{\mu}{\Lambda}\frac{\alpha}{\alpha\beta_2}(\frac{\Lambda}{\mu}-T_1-t_1)+\frac{\alpha+\beta_1}{\alpha+\beta_2}\frac{\mu}{\Lambda}(0-t_2)=1.$$ which both follow the same line of analysis as in the previous case. The social cost was already discussed in case 2.\\
   Case 4: $T_1> A_3$ means that $T_1$ is large enough that the first to arrive after the atom will get service prior to time 0. Recall that $t_3$ is the first time customers arrive after the atom, and arriving after $t_4$ means that service will be granted after time 0. In fact, $A_3$ is the value of $T_1$ such that $t_3=t_4$. $T_1<A_4$ implies that customers cannot arrive as they would have under equilibrium when there are no constraints at all. Consequently, they start arriving from the first point in time they are allowed (i.e. $-T_{e_1}=-A_4$). Otherwise, is they start arriving after $-A_4$, then clearly, one can gain by arriving exactly at $-A_4$. $p_2$, $t_3$ and $t_4$ form the solution of the equations $$\beta_1T_1+\frac{\alpha p_2\Lambda}{2\mu}=-\beta_1t_3+\alpha(\frac{p_2 \Lambda}{\mu}-T_1-t_3),$$ $$-\beta_1 t_3+\alpha(\frac{p_2 \Lambda}{\mu}-T_1-t_3)=-t_4(\beta_1+\alpha),$$ and $$p_2+(t_4-t_3)\frac{\alpha+\beta_1}{\alpha}\frac{\mu}{\Lambda} +\frac{\alpha+\beta_1}{\alpha+\beta_2}\frac{\mu}{\Lambda}(0-t_4)+\frac{\mu}{\Lambda}\frac{\alpha}{\alpha+\beta_2}(\frac{\Lambda}{\mu}-T_1) =1.$$
    Again, the equations arise from the same considerations as in the previous cases. The social cost was already discussed in case 2.
\end{proof}
\begin{thm}
 If $T_1\!+\!T_2\geq  \frac{\Lambda}{\mu}$ and $T_1<T_{e_1}$, then the unique socially optimal strategy has a uniform density of $\frac{\mu}{\Lambda}$ along $[-T_1,\frac{\Lambda}{\mu}-T_1]$, and its overall cost is $\frac{1}{2}\mu(\beta_1T_1^2+\beta_2(T_1-\frac{\Lambda}{\mu})^2)$.
\end{thm}
\begin{proof}
To obtain the socially optimal strategy, recall from
Theorem~\ref{thm:fluidnocons}, when there were no arrival
constraints, that customers' arrival strategy starts at
$t=-\frac{\beta_2\Lambda}{(\beta_1+\beta_2)\mu}$ for both the
equilibrium and the socially optimal strategies. Clearly, now that
$T_1<T_{e_1}$, customers will arrive later. Of course, within the
new constraint customers will arrive as early as they can, which is
time $-T_1$. Similar to Theorem~\ref{thm:fluidnocons}, customers
will arrive at a rate $\mu$ in order to avoid waiting cost, with any
other rate being suboptimal. After deriving the social optimal strategy, the overall \added{social} cost can be computed accordingly.
Whereas, it is obtained by,
 \begin{align*}
  -\int_{-T_1}^{0}\beta_1\mu tdt+\int_{0}^{\frac{\Lambda}{\mu}-T_1}\beta_2\mu
  tdt.
  \end{align*}
\end{proof}
\begin{remark}
The PoA in case 1 equals $\frac{\alpha\Lambda^2+2T_1\beta_1\Lambda\mu+\beta_2(\Lambda-T_1\mu)^2} {T_1^2\beta_1\mu^2+\beta_2(\Lambda-T_1\mu)^2}$. The PoA in cases 2, 3 and 4 equals $\frac{2\beta_2\Lambda(\Lambda-T_1\mu)}{T_1^2\beta_1\mu^2+\beta_2(\Lambda-T_1\mu)^2}$. One can see that $\alpha$ has no impact on the PoA in these cases. This is a surprising result, because drops under equilibrium pay a waiting cost.
\end{remark}
We next deal with the case where only the closing time affects customers' strategy.
\begin{thm}\label{thm:fluidT2}
If  $T_1+T_2\geq \frac{\Lambda}{\mu}$ and $T_2<T_{e_2}$, then the equilibrium arrival strategy is
\begin{align*}
  f(t) = \begin{cases}
\frac{\alpha+\beta_1}{\alpha}\frac{\mu}{\Lambda},  &\hspace{0.8cm}  -\frac{(\alpha+\beta_2)\Lambda-T_2\alpha\mu}{(\alpha+\beta_1+\beta_2)\mu}< t <\frac{\beta_1(-\alpha\Lambda-\beta_2\Lambda+T_2\alpha\mu)}{(\alpha+\beta_1)(\alpha+\beta_1+\beta_2)\mu},\\
\frac{\alpha+\beta_1}{\alpha+\beta_2}\frac{\mu}{\Lambda},     &  -\frac{\beta_1(-\alpha\Lambda-\beta_2\Lambda+T_2\alpha\mu)}{(\alpha+\beta_1)(\alpha+\beta_1+\beta_2)\mu}\leq  t \leq 0 ,  \\
\frac{\alpha}{\alpha+\beta_2}\frac{\mu}{\Lambda},       &
\hspace{3.05cm} 0\leq  t \leq T_2.
                          \end{cases}
\end{align*}
  Under equilibrium, the social cost is $\Lambda\beta_1\frac{(\alpha+\beta_2)\Lambda-T_2\alpha\mu}{(\alpha+\beta_1+\beta_2)\mu}$.
\end{thm}
\begin{proof}
We follow the same line of analysis as in Theorem~\ref{thm:fluidnocons}, where now $T_{e_2}$ is replaced by a fixed value of $T_2$ in the equilibrium and socially optimal arrival strategies, respectively.
\end{proof}
\begin{thm}
If  $T_1+T_2\geq \frac{\Lambda}{\mu}$ and $T_2<T_{e_2}$, then the unique socially optimal arrival strategy has a density of $\frac{\mu}{\Lambda}$ along $[-\frac{\alpha\Lambda+\beta_2\Lambda-T_2\alpha\mu}{(\alpha+\beta_1+\beta_2)\mu},T_2]$, and a mass of $\Lambda-(\frac{\alpha\Lambda+\beta_2\Lambda-T_2\alpha\mu}{(\alpha+\beta_1+\beta_2)\mu}+T_2)\mu$ at $T_2$. The socially optimal strategy overall cost is
\begin{align}
\frac{\beta_1(\alpha+\beta_2)\Lambda^2-2T_2\alpha\beta_1\Lambda\mu+T_2^2\alpha(\beta_1+\beta_2)\mu^2}{2(\alpha+\beta_1+\beta_2)\mu}.
\label{eq:T2finiteoptiamloverallcost}
\end{align}\label{thm:optiamlT2finite}
\end{thm}
\begin{proof}
Clearly, arriving at any other rate than $\mu$ is suboptimal. Moreover, we construct the socially optimal cost function such that we allow a mass at $T_2$. This is because it can be less expensive for a mass to arrive at time $T_2$, rather than arriving earlier than everyone else. Let $-L\!B$ be the lower bound of the socially optimal arrival support. The socially optimal cost function is
\begin{align}
\int_{-L\!B}^0-\beta_1 t\mu dt\! +\!\int_{0}^{T_2} \beta_2 t\mu
dt\!+\!(\Lambda\!-\!(LB\!+\!T_2)\mu)\frac{(\Lambda\!-\!(L\!B\!+\!T_2)\mu)}{2\mu}(\alpha\!+\!\beta_2).
\label{eq:optiamlT2LB}
\end{align}
Finally, $L\!B$ is the value that minimizes~(\ref{eq:optiamlT2LB}).
\end{proof}
\begin{remark}
The PoA is $\frac{2\beta_1\Lambda((\alpha+\beta_2)\Lambda-T_2\alpha\mu)} {\beta_1(\alpha+\beta_2)\Lambda^2-2T_2\alpha\beta_1\Lambda\mu+T_2^2\alpha(\beta_1+\beta_2)\mu^2}$. We observe that if $T_2=0$, the PoA equals exactly 2.
\end{remark}
\noindent We next address the case where $(T_1+T_2)<\frac{\Lambda}{\mu}$. Different shapes of equilibrium arrival strategies are formed by different set of values of $(T_1,T_2)$. Similarly to Theorem~\ref{thm:fluidT1}, we split these into 4 cases for different set of values of $T_1$, only now the thresholds that distinguish them are effected by the value of $T_2$. Before presenting the equilibrium arrival strategy we define several points in time and probabilities, which play a role in the specific structure of the equilibrium.
\begin{align*}
&A'_1=\frac{-\alpha\Lambda-\beta_1\Lambda+\sqrt{\Lambda}\sqrt{((\alpha +\beta_1)^2+\alpha\beta_2+\beta_2^2) \Lambda-2T_2\beta_2\alpha\mu}}{\beta_2\mu},\\
&A'_2=\frac{-\beta_1\beta_2\Lambda-\alpha(\beta_1+\beta_2)\Lambda+T_2\alpha\beta_1\mu +\alpha^2(T_2\mu-\Lambda)}{\beta_2(\alpha+\beta_2)\mu}+\\ &\frac{\sqrt{((\alpha+\beta_2)^2+\alpha\beta_2+\beta_2^2)((\alpha+\beta_2)\Lambda-T_2\alpha\mu)^2}}{\beta_2(\alpha+\beta_2)\mu},\\
&A'_3=\frac{2(\alpha+\beta_2)\Lambda-2T_2\alpha\mu}{(3\alpha+2(\beta_1+\beta_2))\mu},
 \quad A_4=\frac{\Lambda}{\mu}-T_2,\\
&p'_1=\frac{(\alpha\!+\!\beta_2)\Lambda^2\!-\!(T_2\alpha\!+\!T_1(\alpha\!+\!\beta_1))\Lambda\mu}{(\alpha\!+\!\beta_2)\Lambda^2}+\\&\frac{ \sqrt{\Lambda^2(\!-\!T_1^2\beta_2(\alpha\!+\!\beta_2)\mu^2\!+\!((\alpha\!+\!\beta_2)\Lambda\!- \!(T_2\alpha\!+\!T_1(\alpha\!+\!\beta_1))\mu)^2)}}{(\alpha\!+\!\beta_2)\Lambda^2},\\
&p'_2=\frac{2(\alpha+\beta_2)\Lambda-2(T_2\alpha+T_1(\alpha+\beta_1+\beta_2))\mu}{\alpha\Lambda},\\
&t'_1\!\!=\!\frac{-T_1(\alpha\!\!+\!\!\beta_1)\Lambda\mu\!+\!\!\sqrt{\Lambda^2(\!-\!T_1^2\beta_2(\alpha\!+\!\beta_2)\mu^2\!\!+\!\!((\alpha\!+\!\beta_2)\Lambda \!\!-\!\!(T_2\alpha\!+\!T_1(\alpha\!+\!\beta_1))\mu)^2}}{\alpha\Lambda\mu},\\
&t'_2=-T_1\!\!+\!\frac{\sqrt{\Lambda^2(-T_1^2\beta_2((\alpha\!+\!\beta_2)\mu^2\!+\! ((\alpha\!+\!\beta_2)\Lambda)\!-\! (T_2\alpha\!+\!T_1(\alpha\!+\!\beta_1))\mu)^2)}}{(\alpha\!+\!\beta_1)\Lambda\mu},\\
&t'_3=\frac{\Lambda(\alpha+\beta_2)-(T_2\alpha+T_1(2(\alpha+\beta_1)+\beta_2))\mu} {(\alpha+\beta_1)\mu}, \\ &t'_4=\frac{-(\alpha+\beta_2)\Lambda+(T_2\alpha+T_1(\alpha+\beta_2))\mu} {(\alpha+\beta_1)\mu}.
\end{align*}
\begin{thm}\label{thm:fluidT1T2}
If  $T_1+T_2<\frac{\Lambda}{\mu}$, the equilibrium arrival strategy is as follows.\\
 Case 1: If $T_1\leq A'_1$ then a pure equilibrium is formed at time $-T_1$. The corresponding social cost is \begin{align}
 \Lambda\left(\beta_1T_1+\frac{\alpha
 \Lambda}{2\mu}+\beta_2\frac{\mu(\frac{\Lambda}{\mu}-T_1)^2}{2\Lambda}\right).
 \label{eq:socialcosteqcase1t1finite}\end{align}
 Case 2: If $ A'_1< T_1\leq A'_2 $ the arrival strategy incorporates an atom with size $p'_1$ at time $-T_1$, followed by a gap until time $t'_1>0$. Then there is a density of $\frac{\alpha}{\alpha+\beta_2}\frac{\mu}{\Lambda}$ along $[t'_1,T_2]$. The corresponding social cost is \begin{align}
 \Lambda\left(T_2\beta_2+(\alpha+\beta_2)\left(\frac{\Lambda}{\mu}-(T_1+T_2)\right)\right).\label{eq:eqsocialcostT1finitecase234}
   \end{align}
   In the next two cases the social cost is the same.\\
 Case 3: If $ A'_2< T_1\leq A'_3$ then the arrival strategy incorporates an atom of size $p'_1$ (the same size as the previous case) at time $-T_1$, followed by a gap until time $t'_2<0$. Then there is a positive density of $\frac{\alpha+\beta_1}{\alpha+\beta_2}\frac{\mu}{\Lambda}$ along $[t'_2,0]$, and finally a positive density of $\frac{\alpha}{\alpha+\beta_2}\frac{\mu}{\Lambda}$ along $[0, T_2]$.\\
  Case 4: If $ A'_3< T_1\leq A'_4$ then the arrival strategy incorporates an atom with size $p'_2$ at time $-T_1$, followed by a gap until time $t'_3<0$. Then there is a positive density of $\frac{\alpha+\beta_1}{\alpha}\frac{\mu}{\Lambda}$ along $[t'_3,t'_4]$, then a density of $\frac{\alpha+\beta_1}{\alpha+\beta_2}\frac{\mu}{\Lambda}$ along $[t'_4,0]$, followed by another density of $\frac{\alpha}{\alpha+\beta_2}\frac{\mu}{\Lambda}$ along $[0, T_2]$.
\end{thm}
\begin{proof}
We follow the same line of analysis as for Theorem~\ref{thm:fluidT1}, where now $\frac{\Lambda}{\mu}-T_1$ is replaced by a fixed value of $T_2$ regarding the equilibrium strategy and its social cost.
\end{proof}
\begin{thm} \label{thm:T1finiteT2finiteoptimal}
One of the following exhaustive and mutually exclusive cases occurs.\\
 Case \textbf{\romannum{1 }}: If $T_1\leq \frac{\alpha\Lambda+\beta_2\Lambda-T_2\alpha\mu}{(\alpha+\beta_1+\beta_2)\mu}$, then the unique socially optimal strategy has a uniform density of $\frac{\mu}{\Lambda}$ along $[-T_1,T_2]$ and then a mass of $\Lambda-\mu(T_2+T_1)$ at time $T_2$. The overall cost of the socially optimal strategy is
  \begin{align}\frac{1}{2}(T_1^2\beta_1\mu\!+\!T_2^2\beta_2\mu\!+\!2\beta_2(\Lambda\!-\!(T_1+T_2)\mu) +\frac{(\alpha+\beta_2)(\Lambda-(T_1+T_2)\mu)^2}{\mu}).\label{eq:socialoptimalT1finite1}\end{align}
 Case \textbf{\romannum{2 }}: If $T_1>\frac{\alpha\Lambda+\beta_2\Lambda-T_2\alpha\mu}{(\alpha+\beta_1+\beta_2)\mu}$, then
the unique socially optimal strategy coincides with the one
presented in Theorem~\ref{thm:optiamlT2finite}.
\end{thm}

\begin{proof}
Case \emph{\textbf{\romannum{1 }}}:
The socially optimal strategy is as follows: those who arrive during the interval $[-T_1,T_2]$ do not pay a waiting cost at all, while those who arrive prior to time 0 pay an earliness cost, and those who arrive after time 0 pay a tardiness cost. Within this group, for which the total mass is $(T_2+T_1)\mu$, the overall earliness cost is $\frac{1}{2}T_1^2\beta_1\mu$ and the overall tardiness cost is $\frac{1}{2}T_2^2\beta_2\mu$. There is a mass of $\Lambda-(T_2+T_1)\mu$ which arrives at time $T_2$. Each drop immediately pays $\beta_2T_2$. Moreover, on average each one waits $\frac{\frac{\Lambda}{\mu}-(T_1+T_2)}{2}$ time units, which includes both waiting and tardiness costs. Therefore, the total contribution to the overall cost is $\left(\left(\frac{\frac{\Lambda}{\mu}-(T_1+T_2)}{2}\right)+\beta_2T_2\right)\left(\Lambda-(T_1+T_2)\mu\right)$.\\
Case \emph{\textbf{\romannum{2 }}}:  According to
Theorem~\ref{thm:optiamlT2finite},
$T_1>\frac{\alpha\Lambda+\beta_2\Lambda-T_2\alpha\mu}{(\alpha+\beta_1+\beta_2)\mu}$
implies that the lower bound constraint of the arrival is not
effective under the socially optimal strategy. Therefore, it
coincide  with Theorem~\ref{thm:optiamlT2finite}.
\end{proof}
\begin{remark}
Recall from Theorem~\ref{thm:fluidT1T2} that there are two different
expressions for the equilibrium social welfare, for case 1 and cases
2, 3 and 4. Moreover, there are two different socially optimal
welfare expressions, as shown in
Theorem~\ref{thm:T1finiteT2finiteoptimal}. This means that there are
4 different expressions for the PoA when
$T_1+T_2<\frac{\Lambda}{\mu}$. All four PoA values are obtained by
dividing Eqs.~(\ref{eq:socialcosteqcase1t1finite})
and~(\ref{eq:eqsocialcostT1finitecase234}) by
Eqs.~(\ref{eq:T2finiteoptiamloverallcost})
and~(\ref{eq:socialoptimalT1finite1}).
\end{remark}

\section{Conclusions}\label{con:conclusions}
We have studied the behavior of strategic customers who decide when
to arrive to a queue. Each one wishes to minimize his earliness,
tardiness and waiting costs. The contribution of this study is due to the additional earliness cost to the arrival dilemma.
 As discussed in the introduction, we believe it is vital to consider this cost when deciding when to arrive.
  This is because, in many cases, this reflects much better the situation the customers are facing.
   The results showed that under the assumption that customers are not constrained in their arrival times, the symmetric equilibrium is unique and mixed. The equilibrium behavior for arriving after time 0 is similar to previous studies in which
    earliness cost was not included (see~\cite{Haviv13}). However, the behavior prior to time 0 is quite unusual. If we focus on $f(t)$ for both
    $t<0$ and $t>0$ we observe a big difference and it is that, beside the model parameters, for $t>0$, $f(t)$ is affected only by $p_0(t)$, but
    for $t<0$ it is affected by all $p_k(t)$ for $k\in\{0,1,2,...\}$. Moreover, it is also different from the early birds arrival distribution
    from~\cite{Haviv13},~\cite{Ravner14} and \cite{Kliner&hassin11}, in which it was uniform. This is a direct result of the cost function for $t<0$.
     Specifically, the expression regarding the expected tardiness time. Now, customers are affected not only by the expected waiting time,
      but also by the probability of being late. It is also interesting that,
        $f(t)$ can have both minimum and maximum points in the interval of $[-T_{e_1},0]$, as it is shown in the numerical examples. This indicates how different $f(t)$ is compared to previous studies, in which it was either monotone or unimodal.\\
        \indent We discussed the impact of arrival time constraints.\ By imposing pre-determined opening and closing times, the resulting equilibrium strategy can be either pure, mixed with an atom,
        or mixed without atoms. We note that only when the opening time constraint is effective, there can be an atom. Further more, the two constraints are not independent and this could have a major consequence on the equilibrium behavior. For example, by constraining the closing time, this potentially could change an opening time constraint from being ineffective to be effective. \\
        \indent Guidelines for a numerical procedure to compute the equilibrium strategy with some examples were given, in both the constrained and unconstrained models. Beside the results in the stochastic model, we also derived the equilibrium behavior in the fluid model,
  where, now it is possible to get analytic results as opposed to the stochastic model. This allows us to obtain the socially optimal arrival strategy,
  and consequently the price of anarchy. Although, we are unable to connect the fluid model to the stochastic model,
  we observe that in all the numerical examples the value of $T_{e_1}$ is always larger in the stochastic model.
  That is, customers are inclined to start arriving earlier. We believe that this is due to the uncertainty of the queue length in the stochastic model.
   More specifically, due to the variance of the workload,
   if one decides to arrive prior to time 0 then there is a positive probability he will pay tardiness cost.
   This is not the case in the fluid model, where one knows exactly if he pays tardiness cost or not.
   This uncertainty, will eventually leads customers to start arriving earlier.
   This is helpful, as it may used as a lower bound of the arrival point. Moreover, this is also true for the atom size $p$,
    in a constrained model. That is, more customers will arrive at the atom in the stochastic model from the same reasons.\\
     \indent We next discuss two limitations of our model. The first limitation is that we assumed linear costs.
      Of course, non-linear costs are much more realistic. However, considering non-linear cost functions complicates things significantly and the type of the
       equations may alter the result completely. For future study, we suggest considering perhaps some form of convex cost function,
       where it describes better real life situation. As mentioned, non-liner costs were considered in~\cite{Breinbjerg}.\ The second limitation is that the customers are homogenous with respect to their cost function.
        However, heterogeneous customers is clearly more realistic. We note that considering heterogeneous customers can be done only
         by using the fluid model and otherwise is too complicated. It has been done in~\cite{jainhonnappa12} and~\cite{junejajainshimkin11},
         where both waiting and tardiness costs were considered under a fluidic model. It would be interesting, for future study to consider heterogeneous users as well.
          Another issue we have yet to attend to is the existence of an asymmetric equilibria. In~\cite{junejashimkin11}, where both waiting and tardiness were considered, they first proved that there is a unique Nash equilibrium (not necessarily symmetric at this point), and then they proved that it is symmetric. By this, they ruled out asymmetric Nash equilibria. We believe that this is the case in our model as well, due to the similarity of our models and~\cite{junejashimkin11}. Whereas, our model is an extension of the one in~\cite{junejashimkin11}.

\blank{4cm}

\renewcommand{\abstractname}{Acknowledgements}
\begin{abstract}
Thanks are due to Tomer Sharon for generating numerical examples and
Joseph Kreimer for making helpful comments. Also, we wish to thank
the editor and reviewers for their thorough work and meaningful
suggestions. This research is supported by the Israel Science
Foundation, grant No. 1319/11.
\end{abstract}

\bibliographystyle{plain}

\begin{thebibliography}{10}

\bibitem{Breinbjerg}
J.~Breinbjerg.
\newblock Equilibrium Arrival Times to Queues with General
Service Times and Non-Linear Utility Functions.
\newblock {\em European Journal
of Operational Research}, 261(2):595--605, 2017.


\bibitem{Breinbjergsebald}
J.~Breinbjerg, A~.Sebald, and L.P.~ØSterdal.
\newblock Strategic behavior
and social outcomes in a bottleneck queue: experimental evidence.
\newblock {\em Review of
 Economic Design}, 20(3):207--236, 2016.

\bibitem{Glazer&hassin83}
R.~Hassin, and A.~Glazer.
\newblock ?/m/1: on the equilibrium distribution of customer arrivals.
\newblock {\em European Journal of Operational Research}, 13:146--175, 1983.


\bibitem{Haviv13}
M.~Haviv.
\newblock When to arrive at a queue with tardiness costs.
\newblock {\em Performance Evaluation}, 70:387--399, 2013.

\bibitem{Arnottpalma99}
A.~Lago, and C.~F. Daganzo.
\newblock Information and time-of-usage decisions in the bottleneck model with
  stochastic capacity and demand.
\newblock {\em European Economic Review}, 43:525--548, 1999.

\bibitem{LagoDaganzo07}
A.~Lago, and C.~F. Daganzo.
\newblock Spillovers, merging traffic and the morning commute.
\newblock {\em Transportation Research Part B}, 41:670--683, 2007.

\bibitem{Havivmilchtaich2012}
I.~Milchtaich, and M.~Haviv.
\newblock Auctions with a random number of identical bidders.
\newblock {\em Economics Letters}, 144:143--146, 2012.

\bibitem{Kerner09}
O.~Kella, M.~Haviv, and Y.~Kerner.
\newblock Equilibrium strategies in queues based on time or index of arrival.
\newblock {\em Probability in the Engineering and Informational Sciences},
  24:13--25, 2010.

\bibitem{OstuboRapoport08}
H.~Ostubo, and A.~Rapoport.
\newblock Vickrey’s model of traffic congestion discretized.
\newblock {\em Transportation Research Part B: Methodological}, 42:873--889,
  2008.

\bibitem{mcafeemcmillan87}
J.~McMillan, and P.~McAfee.
\newblock Auctions with a stochastic number of bidders.
\newblock {\em Journal of Economic Theory}, 43:1--19, 1987.




\bibitem{Platz2012}
T.~Platz, and L.P.~Østerdal.
\newblock The curse of the first-in-first-out queue discipline.
\newblock {\em Games and Economic Behavior}, 104:165–176, 2017.


\bibitem{Kliner&hassin11}
Y.~Kleiner, and R.~Hassin.
\newblock Equilibrium and optimal arrival patterns to a server with opening and
  closing times.
\newblock {\em IIE transactions}, 43:164--175, 2011.

\bibitem{jainhonnappa12}
H.~Honnappa, and R.~Jain.
\newblock Strategic arrivals into queueing networks: the network queueing game.
\newblock {\em Operations research,}, 63:247--259, 2014.



\bibitem{junejashimkin11}
N.~Shimkin, and S.~Juneja.
\newblock The concert queueing game: strategic arrivals with waiting and
  tardiness costs.
\newblock {\em Queueing systems}, 74:369--402, 2013.

\bibitem{junejajainshimkin11}
N.~Shimkin, S.~Juneja, and R.~Jain.
\newblock The concert/cafeteria queuing problem: to wait or to be late.
\newblock {\em Discrete Event Dynamic Systems}, 21:103--138, 2011.

\bibitem{junejajain09}
R.~Jain, and S.~Juneja.
\newblock The concert/cafeteria queuing problem: a game of arrivals.
\newblock {\em Proceeding of the Fourth International ICST Conference on
  Performance Evaluation on Performance methodology}, Valuetools 2009, Pisa,
  Italy, ACM, 2009.

\bibitem{Seale05}
D.A~Seale, J.E.~Parco, W.E.~Stein, and A. Rapoport.
\newblock Joining a
 queue or staying out: Effects of information structure and
service time on arrival and staying out decisions
\newblock {\em Experimental Economics}, 8(2):117--144, 2005.

\bibitem{Stein07}
 W.E.~Stein,  A. Rapoport, D.A~Seale, D.A.~Zhang, and R.~Zwick.
\newblock Batch queues with choice of arrivals: Equilibrium analysis and experimental study.
\newblock {\em Games and Economic Behavior}, 59(2):345--363, 2007.

\bibitem{Rapoport04}
A.~Rapopot, A.~Stein W.E.~Parco, and D.A.~Seale
\newblock Equilibrium play in single-server queues with endogenously determined
arrival times.
\newblock {\em Journal of Economic Behavior and Organization}, 55(1):67--91, 2004.

\bibitem{Ravner14}
L.~Ravner.
\newblock Equilibrium arrival times to a queue with order penalties.
\newblock {\em European Journal of Operational Research}, 239:456--468, 2014.

\bibitem{ravnerhassin17}
A.~Glazer, R.~Hassin, and L.~Ravner
\newblock Equilibrium and efficient clustering of arrival times to a queue
\newblock Archive (https://arxiv.org/abs/1701.04776),  2017.

\bibitem{Vickery69}
W.S. Vickrey.
\newblock Congestion theory and transport investment.
\newblock {\em The American Economic Review}, 59:251--260, 1969.


\end{thebibliography}

\end{document}